\documentclass[11pt]{article}
\usepackage[top=2.5cm,bottom=2.5cm,left=2.5cm,right=2.5cm]{geometry}
\usepackage{mathrsfs}
\usepackage{pstricks}
\usepackage{amsmath,amssymb,graphicx,amsthm,tabularx,array}
\usepackage{hyperref}
\usepackage{makecell}
\theoremstyle{plain}
\newtheorem{theorem}{Theorem}[section]

\newtheorem{lem}[theorem]{Lemma}

\newtheorem{myCli}{Claim}
\newtheorem{pro}[theorem]{Problem}

\theoremstyle{definition}

\newtheorem{other}{}

\title{\bf A signless Laplacian spectral Erd\H{o}s-Stone-Simonovits theorem\footnote{This paper was published on Discrete Mathematics 349 (2026) 114665. This is the final version. E-mail addresses: \url{zhengj@jxnu.edu.cn} (J. Zheng), \url{lhh@jxnu.edu.cn} (H. Li), \url{suli@jxnu.edu.cn} (L. Su).}}

\author{
Jian Zheng, \quad
Honghai Li\footnote{Corresponding author}, \quad  
Li Su, 
\\
\small   School of  Mathematics and Statistics, Jiangxi Normal University,  Nanchang\\
\small  Jiangxi 330022,  China. Email: \url{zhengj@jxnu.edu.cn}, \url{suli@jxnu.edu.cn} }

\date{\today}
\begin{document}
\maketitle
\begin{abstract}
 The celebrated Erd\H{o}s--Stone--Simonovits theorem states that
$\mathrm{ex}(n,F)= \big(1-\frac{1}{\chi(F)-1}+o(1) \big)\frac{n^{2}}{2}$,
where $\chi(F)$ is the chromatic number of $F$. In 2009, Nikiforov proved a spectral extension of the Erd\H{o}s--Stone--Simonovits theorem in terms of the adjacency spectral radius. 
In this paper, we shall establish a unified extension in terms of the signless Laplacian spectral radius. 
Let $q(G)$ be the signless Laplacian spectral radius of $G$ and we denote $\mathrm{ex}_{q}(n,F) =\max \{q(G):|G|=n ~\mbox{and}~F\nsubseteq G\}$. 
It is known that the Erd\H{o}s--Stone--Simonovits type result for the signless Laplacian spectral radius
does not hold for even cycles. 
We prove that if $F$ is a graph with $\chi(F)\geq 3$, then
$\mathrm{ex}_{q}(n,F)=\big(1-\frac{1}{\chi(F)-1}+o(1) \big)2n$.
This solves a problem proposed by Li, Liu and Feng (2022), which gives an entirely satisfactory answer to the problem of estimating $\mathrm{ex}_q(n,F)$.  Furthermore, it extends the aforementioned result of
Erd\H{o}s, Stone and Simonovits as well as the spectral result of Nikiforov.
Our result indicates that the Erd\H{o}s--Stone--Simonovits type result regarding the signless Laplacian spectral radius is valid in general. 

\vspace{3mm}

\vspace{2mm}

\end{abstract}

\section{Introduction}

Let $G$ be a simple graph on $n$ vertices with vertex set $V(G)$ and edge set $E(G)$. We denote by $e(G)$ the number of edges of $G$, that is, $e(G)=|E(G)|$. For a vertex $u\in V(G)$, we write $ N (u)$ for the set of neighbors of $u$, and $G-u$ for the graph obtained from $G$ by removing $u$ from $V(G)$ and removing all edges containing $u$ from $E(G)$, that is, $G-u$ denotes the subgraph of $G$ induced by $V(G)\setminus \{u\}$.
The \emph{degree} $d_{G}(u)$ (or simply $d(u)$)  is the number of edges in $G$ containing $u$.
We write $\delta(G)$ for the \emph{minimum degree} of $G$.
The \emph{adjacency matrix} of $G$ is defined as  $A(G)=[a_{ij}]_{n\times n}$  where $a_{ij}=1$ if two vertices $v_{i}$ and $v_{j}$ are adjacent in $G$ and $a_{ij}=0$ otherwise. The largest eigenvalue of $A(G)$, denoted by $\lambda(G)$, is called the \emph{adjacency spectral radius} of $G$.

A graph $G$ is called \emph{$F$-free} if it does not contain a subgraph isomorphic to $F$. The \emph{Tur\'an number} of $F$,  denoted by $\mathrm{ex}(n,F)$, is defined as the maximum number of edges in an $F$-free graph on $n$ vertices.
The  \emph{Tur\'an density} of a graph $F$ is defined as
$$\pi(F):=\lim\limits_{n\to \infty}\frac{\mathrm{ex}(n,F)}{\binom{n}{2}}.$$
The existence of the limit is a direct consequence of the averaging argument proposed by Katona--Nemetz--Simonovits \cite{KNS1964}.
 Let $T_{r}(n)$ be the $r$-partite Tur\'{a}n graph, which is a complete $r$-partite graph on $n$ vertices where the sizes of its partite set are as equal as possible.
It is easy to verify that $e(T_r(n)) = (1- \frac{1}{r}) \frac{n^2}{2} - O_r(1)$.
The well-known Tur\'{a}n theorem (see \cite{Bol78}) asserts that if $G$ is
an $n$-vertex $K_{r+1}$-free graph, then
$e(G)\le e(T_r(n))$, with equality if and only if $G=T_{r}(n)$. Under the above notation, the Tur\'{a}n theorem implies that $\mathrm{ex}(n,K_{r+1})=e(T_r(n))$ and $\pi (K_{r+1}) = 1- \frac{1}{r}$.

Erd\H{o}s and Stone \cite{ES1946}, and Erd\H{o}s and Simonovits \cite{ES1966} independently established
the following result, which provides an asymptotic value of $\mathrm{ex}(n,F)$ for a general graph $F$.

\begin{theorem}[Erd\H{o}s--Stone--Simonovits~\cite{ES1946,ES1966}]\label{edg}
If $F$ is a graph with chromatic number $\chi(F)\ge 2$, then
$$\mathrm{ex} (n,F)=\bigg(1-\frac{1}{\chi(F)-1}+o(1)\bigg)\frac{n^{2}}{2}.$$
\end{theorem}

Theorem \ref{edg} is equivalent to saying that $\pi(F)=1-\frac{1}{\chi(F)-1}$. In some cases, the Tur\'an density
of a graph is more useful in applications than   the Tur\'an number, which can be seen in our later proof.
In this note, we investigate the spectral extension of Theorem \ref{edg}.
We define $\mathrm{ex}_{\lambda}(n,F)$ to be the maximum of the adjacency spectral radius among all $n$-vertex $F$-free graphs, that is,
$$\mathrm{ex}_{\lambda}(n,F):=\max \{\lambda(G):|G|=n ~\mbox{and}~F\nsubseteq G\}.$$
In 2009, Nikiforov \cite{NS2009}  presented a spectral version of Theorem \ref{edg}.

\begin{theorem}[Nikiforov \cite{NS2009}]\label{adj}
If $F$ is a graph with chromatic number $\chi(F)\ge 2$, then
$$\mathrm{ex}_{\lambda}(n,F)=\bigg(1-\frac{1}{\chi(F)-1}+o(1)\bigg)n.$$
\end{theorem}

 By the Rayleigh quotient, it is well-known that $\lambda (G) \ge 2e(G)/n$ for every $n$-vertex graph $G$. So the spectral version in Theorem \ref{adj} can imply the size version in Theorem \ref{edg}.
An  alternative proof of Theorem \ref{adj} can be found in \cite{LP2022}.

\subsection{The signless Laplacian spectral radius}

 The \emph{signless Laplacian matrix} of a graph $G$ is defined as $Q(G):=D(G)+A(G)$,
where $D(G)$ is the degree diagonal matrix and $A(G)$ is the adjacency matrix of $G$.
  The largest eigenvalue of $Q(G)$, denoted by $q(G)$, is called \emph{$Q$-index} or the \emph{signless Laplacian spectral radius} of $G$. It is easy to see that there is always a nonnegative eigenvector associated to $q(G)$.
   The signless Laplacian matrix  of graphs has been extensively
studied (see surveys  \cite{CS09,CS10,CS10AADM}).
By the Perron--Frobenius theorem, for a connected graph $G$, there exists a positive unit eigenvector corresponding to $q(G)$, which is called the Perron vector of $Q(G)$.
We define $\mathrm{ex}_{q}(n,F)$ to be the maximum signless Laplacian spectral radius of an $F$-free graph on $n$ vertices, that is,
$$\mathrm{ex}_{q}(n,F):=\max \{q(G):|G|=n ~\mbox{and}~F\nsubseteq G\}.$$
 Determining the value of $\mathrm{ex}_q(n,F)$ has a long history in the study of spectral graph theory.
For example, it was shown in \cite{AN2013ela,HJZ2013} that if $G$ is an $n$-vertex $K_{r+1}$-free graph, then $q(G)\le (1- \frac{1}{r})2n$.
Observe that this result involving the signless Laplacian spectral radius can imply the classical Tur\'{a}n theorem.  Many other results on $\mathrm{ex}_q(n,F)$ for various graphs $F$ can be found in the literature; see \cite{Yu2008} for matchings, \cite{NY2014} for paths, \cite{FNP2013,Yuan2014,NY2015} for cycles,
\cite{AFNP2016} for complete bipartite graphs,
\cite{CLZ2020} for linear forests,
\cite{ZHG21} for friendship graphs,
\cite{WZ2023} for fan graphs, and
\cite{CLZ2024} for intersecting odd cycles.

A natural question proposed by Li, Liu and Feng \cite{LLF2022-advances} is to ask whether there is a  signless Laplacian spectral analogue of the Erd\H{o}s--Stone--Simonovits theorem for a general graph $F$.

\begin{pro}[See \cite{LLF2022-advances}] \label{prob-1-3}
Is it true that $\mathrm{ex}_q (n,F)=\big(1-\frac{1}{\chi(F)-1}+o(1)\big)2n$ for any graph $F$?
\end{pro}
Li, Liu and Feng \cite[p. 19]{LLF2022-advances} pointed out that the answer of this question is negative in the case $F=C_{2k+2}$ for every integer $k\geq1$.
 In this note, we shall show that the answer to this problem is positive for any graph $F$ with $\chi(F)\geq3$ (see the forthcoming Theorem \ref{signless}), which could be viewed as a unified extension of both Theorem \ref{edg} and Theorem \ref{adj} by invoking the fact $q(G)\ge 2 \lambda (G) \ge \frac{4m}{n}$.
 Last but not least, we would also like to mention that it was claimed in reference \cite[Theorem 1.3]{AN2012ela} that Nikiforov had solved Problem \ref{prob-1-3} before 2012.
However, his manuscript has never appeared or been made  available.
As Nikiforov has passed away, we are unable to reach his preprint for further comment on this problem.
Inspired by an unpublished work of Nikiforov \cite{N2014A}, we are going to provide a detailed proof of Problem \ref{prob-1-3} in this paper.

\begin{theorem}[Main result] \label{signless}
If  $F$ is a graph with chromatic number $\chi(F)\geq 3$, then
$$\mathrm{ex}_q (n,F)=\bigg(1-\frac{1}{\chi(F)-1}+o(1)\bigg)2n.$$
\end{theorem}

For bipartite graphs $F$, we show the following result.

\begin{theorem} \label{thm-1-5}
If $F$ is a bipartite graph, i.e., $\chi (F)=2$, then
\[   \mathrm{ex}_q(n,F)= \begin{cases} (1+o(1))n & \text{if $F$ is not a star;} \\
2(t-1) &  \text{if $F$ is a star $K_{1,t}$.}
\end{cases} \]
\end{theorem}

\section{Preliminaries}

 It is easy to verify that $q(G)=n$ for any complete bipartite graph $G$ on $n$ vertices, particularly, $q(T_{2}(n))=n$.
The following two lemmas can be found in \cite{AN2012ela,AN2013ela}.
For the sake of completeness, we provide a detailed proof for readers.

\begin{lem} \label{L1}
If $F$ is a graph with chromatic number $\chi(F)\geq3$, then the limit
$$\pi_{q}(F):=\lim\limits_{n\to \infty}\frac{\mathrm{ex}_q (n,F)}{n}$$
exists, and $\pi_{q}(F)$ satisfies
$$\pi_{q}(F)\leq \frac{\mathrm{ex}_q (n,F)}{n-1}.$$
\end{lem}

\begin{proof}
  Suppose that $H$ is an $F$-free graph on $n$ vertices with $q(H)=\mathrm{ex}_q (n,F)$, and $\mathbf{x}=(x_{1},\ldots,x_{n})$ is a nonnegative unit eigenvector  to $q(H)$. Let $\mu_n=\min\{x_{1},\ldots,x_{n}\}$ and $u$ be a vertex for which $x_{u}=\mu_n$.
First, the Rayleigh quotient implies that
\begin{eqnarray} \notag
q(H)&=&\mathbf{x}^\mathrm{T}Q(H)\mathbf{x}=
\sum_{ij\in E(H)}(x_{i}+x_{j})^{2}=\sum_{ij\in E(H-u)}(x_{i}+x_{j})^{2}+\sum_{j\in  N (u)}(\mu_n+x_{j})^{2}\\ \notag
&=& \sum_{ij\in E(H-u)}(x_{i}+x_{j})^{2}+d(u)\mu_n^{2}+2\mu_n\sum_{j\in  N (u)}x_{j}+\sum_{j\in  N (u)}x_{j}^{2}\\
&\leq& (1-\mu_n^{2})q(H-u)+d(u)\mu_n^{2}+2\mu_n\sum_{j\in  N (u)}x_{j}+\sum_{j\in  N (u)}x_{j}^{2}. \label{e1}
\end{eqnarray}
From the eigenvalue-eigenvector equation  for eigenpair ($q(H)$, $\mathbf{x}$) at the vertex $u$, we have
$$(q(H)-d(u))\mu_n=\sum_{j\in  N (u)}x_{j}$$
and then
\begin{eqnarray*}
d(u)\mu_n^{2}+2\mu_n\sum_{j\in  N (u)}x_{j}+\sum_{j\in  N (u)}x_{j}^{2}
&=& d(u)\mu_n^{2}+2(q(H)-d(u))\mu_n^{2}+\sum_{j\in  N (u)}x_{j}^{2}\\
&\leq& d(u)\mu_n^{2}+2(q(H)-d(u))\mu_n^{2}+1-(n-d(u))\mu_n^{2}\\
&=& 2q(H)\mu_n^{2}-n\mu_n^{2}+1.
\end{eqnarray*}
Combining this with (\ref{e1}), we find that
\begin{equation}\label{e2}
q(H-u)\geq q(H)\frac{1-2\mu_n^{2}}{1-\mu_n^{2}}-\frac{1-n\mu_n^{2}}{1-\mu_n^{2}}.
\end{equation}
Since $T_{2}(n)$ is $F$-free, we have $\mathrm{ex}_q (n,F)\geq q(T_{2}(n))=n$. Therefore, by (\ref{e2}),
\begin{eqnarray}\notag
\frac{\mathrm{ex}_q (n-1,F)}{n-2} \geq
\frac{q(H-u)}{n-2}
&\geq& \frac{q(H)}{n-1}\bigg(1+\frac{1}{n-2}\bigg)\frac{1-2\mu_n^{2}}{1-\mu_n^{2}}-\frac{1-n\mu_n^{2}}{(n-2)(1-\mu_n^{2})} \\  \label{zy1}
&=& \frac{\mathrm{ex}_q (n,F)}{n-1}\bigg(1+\frac{1-n\mu_n^{2}}{(n-2)(1-\mu_n^{2})}\bigg)-\frac{1-n\mu_n^{2}}{(n-2)(1-\mu_n^{2})}\\ \notag
&\geq& \frac{\mathrm{ex}_q (n,F)}{n-1}.
\end{eqnarray}
This implies that the sequence $\big\{\frac{\mathrm{ex}_q (n,F)}{n-1}\big\}_{n=1}^{\infty}$ is nonincreasing, and so it is convergent.
\end{proof}

\begin{lem} \label{lem-2-2}
Let $G_{n}$ be an $F$-free graph on $n$ vertices with $q(G_{n})=\mathrm{ex}_{q}(n,F)$. Suppose that $\mathbf{x}=(x_{1},\ldots,x_{n})$ is a nonnegative unit eigenvector  corresponding to $q(G_{n})$. We denote $\mu_{n}=\min\{x_{1},\ldots,x_{n}\}$ and $\delta_{n}=\delta(G_{n})$.
If $\mu_{n}>0$, then $\mathrm{ex}_q(n,F) \leq \delta_{n}+\sqrt{\delta_{n}^{2}+(\frac{1}{n\mu_{n}^{2}}-1)n\delta_{n}}$.
\end{lem}

\begin{proof}
Suppose that $u\in V(G_{n})$ is a vertex such that $d(u)=\delta_{n}$. Then
$$q(G_{n})x_{u}=\delta_{n}x_{u}+\sum_{i\in N (u)}x_{i}.$$
By the Cauchy--Schwarz inequality, we have
\begin{displaymath}
\begin{split}
(q(G_{n})-\delta_{n})^{2}\mu_{n}^{2}& \leq  (q(G_{n})-\delta_{n})^{2}x_{u}^{2}=\bigg(\sum_{i\in N (u)}x_{i}\bigg)^{2}
\leq\delta_{n}\sum_{i\in N (u)}x_{i}^{2}\leq\delta_{n}\bigg(1-\sum_{i\in V(G_{n})\backslash N (u)}x_{i}^{2}\bigg)\\
& \leq \delta_{n}(1-(n-\delta_{n})\mu_{n}^{2})=\delta_{n}-(n\delta_{n}-\delta_{n}^{2})\mu_{n}^{2}.
\end{split}
\end{displaymath}
This implies that
$$ q(G_{n})\leq \delta_{n}+\sqrt{\delta_{n}^{2}+\bigg(\frac{1}{n\mu_{n}^{2}}-1\bigg) n \delta_{n}},$$
completing the proof.
\end{proof}

Even though our proof of Theorem \ref{signless} is relatively short, to start with, we give a detailed overview of the proof   highlighting certain novel aspects of the proof.
Note that
showing $\mathrm{ex}_q(n,F) = 2\left(1- \frac{1}{\chi (F) -1} + o(1) \right)n$ is equivalent to showing $\pi_q(F) = 2\pi (F)$.
 Clearly, we know from Theorem \ref{edg} that $\frac{\delta_n}{n} \le \frac{2e(G_n)}{n^2} \le 1- \frac{1}{\chi (F) -1} + o(1) = \pi(F) +o(1)$.
Using the upper bound in Lemma \ref{lem-2-2},
it suffices to find infinitely many positive integers $n\in \mathbb{N}$ such that $n \mu_{n}^2 \ge 1- o(1)$. If this can be achieved, then Lemma \ref{lem-2-2}
entails $\frac{\mathrm{ex}_q(n,F)}{n} \le \frac{\delta_n}{n} +\sqrt{(\frac{\delta_n}{n})^2 + o(1)\frac{\delta_n}{n}} = 2\frac{\delta_n}{n} + o(1)$, as needed.

The following theorem establishes a key relation between the Tur\'{a}n density and the $q$-spectral  radius density.
Theorem \ref{L2} plays significant role in the proof
of Theorem \ref{signless}.
In the proof of Theorem \ref{L2}, we will adopt some
analytic techniques from the proof of \cite[Theorem $12$]{N2014A}.

\begin{theorem} \label{L2}
If $F$ is a graph with chromatic number $\chi(F)\geq3$ and $\pi_{q}(F)>1$, then
\[  \pi_{q}(F)=2\pi(F). \]
\end{theorem}

\begin{proof}
First of all, we show that $\pi_{q}(F)\geq 2\pi(F)$.
Let $H$ be an $F$-free graph on $n$ vertices with maximum number of edges, that is,
$e(H)=\mathrm{ex}(n,F)$. Taking the vector $\mathbf{y}=(n^{-1/2},\ldots,n^{-1/2}) \in \mathbb{R}^n$, we have
\begin{equation} \label{eq-geq}
\frac{\mathrm{ex}_q (n,F)}{n}\geq \frac{q(H)}{n}  \geq \frac{\mathbf{y}^\mathrm{T}Q(H)\mathbf{y}}{n}=\frac{4\, \mathrm{ex}(n,F)}{n^{2}}.
\end{equation}
Consequently, we get
$\pi_{q}(F)\geq 2\pi(F)$.

In what follows, we show that $\pi_{q}(F)\leq 2\pi(F)$.
Since $\pi_{q}(F)>1$ and obviously $\pi_{q}(F)\leq2$, we assume that $\pi_{q}(F)=1+\varepsilon$, where $0<\varepsilon\leq1$.
Then we have the following two claims.

\begin{myCli}\label{claim-1}
There exist infinitely many $n\in \mathbb{N}$ such that
$$\frac{\mathrm{ex}_q (n-1,F)}{n-2}-\frac{\mathrm{ex}_q (n,F)}{n-1}<\frac{1}{n\log n}\cdot\frac{\mathrm{ex}_q (n,F)}{n-1}.$$
\end{myCli}

\begin{proof}[Proof of Claim \ref{claim-1}]
Assume for a contradiction that there exists $n_{0}$ such that for all $n\geq n_{0}$,
$$\frac{\mathrm{ex}_q (n-1,F)}{n-2}-\frac{\mathrm{ex}_q (n,F)}{n-1}\geq\frac{1}{n\log n}\cdot\frac{\mathrm{ex}_q (n,F)}{n-1}.$$
 Summing the inequalities for all $n_0,n_0+1,\ldots, k$, we get
\begin{eqnarray*}
\frac{\mathrm{ex}_q (n_{0}-1,F)}{n_{0}-2}-\frac{\mathrm{ex}_q (k,F)}{k-1}
&=& \sum_{n=n_{0}}^{k}
\left( \frac{\mathrm{ex}_q (n-1,F)}{n-2}-\frac{\mathrm{ex}_q (n,F)}{n-1} \right) \\
&\geq& \sum_{n=n_{0}}^{k} \frac{1}{n\log n}\cdot\frac{\mathrm{ex}_q (n,F)}{n-1}\\
&\geq& \pi_{q}(F)\sum_{n=n_{0}}^{k} \frac{1}{n\log n},
\end{eqnarray*}
where the last inequality follows from Lemma \ref{L1}. Note that the left-hand side is bounded and the right-hand side
diverges. Taking $k$ sufficiently large, we obtain a contradiction.
\end{proof}

\begin{myCli}\label{claim-2}
For  sufficiently large $n$, if
\begin{equation}\label{zy2}
\frac{\mathrm{ex}_q (n-1,F)}{n-2}<\frac{\mathrm{ex}_q (n,F)}{n-1}\bigg(1+\frac{1}{n\log n}\bigg),
\end{equation}
then
$$\mu_{n}^{2}>\frac{1}{n}\bigg(1-\frac{1}{\log \log n}\bigg).$$
\end{myCli}

\begin{proof}[Proof of Claim \ref{claim-2}]
Assume for a contradiction that $$\mu_{n}^{2}\leq\frac{1}{n}\bigg(1-\frac{1}{\log \log n}\bigg).$$
Similar to (\ref{zy1}), we deduce that
\begin{eqnarray} \notag
\frac{\mathrm{ex}_q (n-1,F)}{n-2}
&\geq & \frac{\mathrm{ex}_q (n,F)}{n-1}\bigg(1+\frac{1-n\mu_{n}^{2}}{(n-2)(1-\mu_{n}^{2})}\bigg)-\frac{1-n\mu_{n}^{2}}{(n-2)(1-\mu_{n}^{2})}\\ \notag
&\geq & \frac{\mathrm{ex}_q (n,F)}{n-1}\bigg(1+\frac{\varepsilon}{1+\varepsilon}\cdot\frac{1-n\mu_{n}^{2}}{(n-2)(1-\mu_{n}^{2})}\bigg)\\ \notag
&\geq&  \frac{\mathrm{ex}_q (n,F)}{n-1}\bigg(1+\frac{\varepsilon(1-n\mu_{n}^{2})}{2n}\bigg)\\
&\geq&  \frac{\mathrm{ex}_q (n,F)}{n-1}\bigg(1+\frac{\varepsilon}{2n\log\log n}\bigg), \label{zy3}
\end{eqnarray}
where the second inequality follows from $\mathrm{ex}_q (n,F)/(n-1)\geq \pi_{q}(F)= 1+\varepsilon$,
and the third inequality follows from $0<\varepsilon\leq1$.
Combining  (\ref{zy2}) with (\ref{zy3}), for sufficiently large $n$, we have
$$\frac{\varepsilon}{2n\log\log n}<\frac{1}{n\log n},$$
which leads to a contradiction, since $\varepsilon >0$ is fixed and $n$ is sufficiently large.
\end{proof}

Combining Claim \ref{claim-1} with Claim \ref{claim-2}, there exists an increasing infinite  sequence   $\{n_{i}\}_{i=1}^{\infty}$ of positive integers such that
for each $n\in \{n_1,n_2,\ldots \}$, we have
\[  \mu_{n}^{2}>\frac{1}{n} \left(1-\frac{1}{\log \log n} \right) >0.\]
Lemma \ref{lem-2-2} implies that
\begin{eqnarray*}
\frac{\mathrm{ex}_q (n,F)}{n} &\leq&  \frac{\delta_{n}}{n}+\sqrt{ \left(\frac{\delta_{n}}{n}\right)^2+\bigg(\frac{1}{n\mu_{n}^{2}}-1\bigg)\frac{\delta_{n}}{n}}\\
&\leq& \frac{\delta_n}{n}+
\sqrt{\left(\frac{\delta_n}{n}\right)^{2} +\bigg(\frac{1}{1-1/ \log \log n}-1\bigg)\frac{\delta_{n}}{n}}\\
&\leq& \frac{2\delta_n}{n}+ \sqrt{\frac{1}{1-1/\log \log n}-1}.
\end{eqnarray*}
Since $\delta_n \le \frac{2e(G_n)}{n} \le \frac{2\, \mathrm{ex}(n,F)}{n}$, it follows that
\[
\pi_{q}(F)=\lim_{i \to \infty}\frac{\mathrm{ex}_q (n_i,F)}{n_i}\leq
\lim_{i \to \infty}\left( \frac{4\, \mathrm{ex}(n_i,F)}{n_i^{2}}
+\sqrt{\frac{1}{1-1/\log \log n_i}-1} \right)=2\pi(F).
\]
This completes the proof of Theorem \ref{L2}.
\end{proof}

\section{Proofs of Theorems \ref{signless} and \ref{thm-1-5}}

The aim of this section is to prove Theorems \ref{signless} and \ref{thm-1-5}.

\begin{proof}[{\bf Proof of Theorem \ref{signless}}]
Due to $\pi (F) = 1- \frac{1}{\chi (F)-1}$ by Theorem \ref{edg},
it is sufficient to show that $\pi_q(F) = 2\pi (F)$ for every $F$ with $\chi (F)\ge 3$.
If $\chi(F)\geq4$, then (\ref{eq-geq}) gives
$\pi_{q}(F)\geq2\pi(F)  \geq\frac{4}{3}$.
By Theorem \ref{L2}, we get $\pi_{q}(F)=2\pi(F)$.
Now, we consider the case $\chi(F)=3$.
Suppose on the contrary that $\pi_{q}(F)>2\pi(F)=1$, then
Theorem \ref{L2} gives $\pi_{q}(F)=2\pi(F)$, which leads to a contradiction.
Thus, we must have $\pi_{q}(F)=2\pi(F)$. This completes the proof of Theorem \ref{signless}.
\end{proof}

Next, we present the proof of Theorem \ref{thm-1-5}.

\begin{proof}[{\bf Proof of Theorem \ref{thm-1-5}}]
Firstly, suppose that $F$ is not a star.
Then the complete bipartite graph $K_{1,n}$ is $F$-free, and it follows that $\mathrm{ex}_q(n,F) \ge q(K_{1,n})=n$.
Since $F$ is not a star, there exists an integer $t\ge 2$ such that $F$ is a subgraph of $K_{t,t}$.
Assume that $G$ is an $n$-vertex $F$-free graph, then $G$ is $K_{t,t}$-free. The K\H{o}vari--S\'{o}s--Tur\'{a}n theorem (see \cite{Bol78})
implies that $e(G)\le \mathrm{ex}(n,K_{t,t})= O(n^{2-1/t})$.
It was proved by Feng and Yu \cite{FY2009} that
if $G$ is a graph on $n$ vertices, then
$ q(G) \le \frac{2e(G)}{n-1} +n-2 \le O(n^{1-1/t}) + n$.
Consequently, we obtain $\mathrm{ex}_q(n,F)\le
(1+o(1))n$. Alternatively, we can derive this bound in a different way. Since $F$
is a bipartite graph, there exists a graph $F'$ such that $\chi(F')=3$ and $F\subseteq F'$.
By Theorem \ref{signless}, we have $\mathrm{ex}_q(n,F)\le \mathrm{ex}_q(n,F')\le
(1+o(1))n$, as desired. 

Suppose that $F$ is a star, say $F=K_{1,t}$.
Since $G$ is $K_{1,t}$-free, we know that the maximum degree $\Delta (G)\le t-1$.
It is well-known \cite[p. 50]{BH2012} that $q(G)\le
\max_{xy\in E(G)} (d(x)+d(y)) \le 2(t-1)$.
Therefore, we get $\mathrm{ex}_q(n,F) \le 2(t-1)$.
Taking $G$ as the disjoint union of  copies of $K_t$ (with some isolated vertices if possible),  we can see that $q(G)=q(K_t)= 2(t-1)$, which yields
$\mathrm{ex}_q(n,F) = 2(t-1)$.
\end{proof}

\noindent \textbf{Availability of data and materials.} No data was used for the research described in the article. The authors have no relevant financial or non-financial interests to
 disclose.

\subsection*{Acknowledgement}  
H. Li and L. Su are supported by National Natural Science Foundation of China (Nos.  12161047,  12061038) and  Jiangxi Provincial Natural Science foundation (No. 20224BCD41001).

\end{document}